\documentclass[11pt,reqno]{amsart}
\usepackage{amsmath,amssymb,amsfonts,amsthm}
\usepackage[margin=1.15in]{geometry}
\usepackage[numbers,sort&compress]{natbib}
\usepackage{hyperref}
\hypersetup{hidelinks,pdftitle={K-area bounds for smooth sphere bundles},pdfauthor={Haoxuan Cheng}}
\newtheorem{theorem}{Theorem}[section]
\newtheorem{corollary}[theorem]{Corollary}
\newtheorem{lemma}[theorem]{Lemma}
\newtheorem{proposition}[theorem]{Proposition}
\theoremstyle{definition}
\newtheorem{definition}[theorem]{Definition}
\newcommand{\R}{\mathbb R}
\newcommand{\C}{\mathbb C}
\newcommand{\Z}{\mathbb Z}
\newcommand{\Karea}{\operatorname{K\text{-}area}}
\newcommand{\ch}{\operatorname{ch}}
\newcommand{\rk}{\operatorname{rk}}
\newcommand{\op}{\mathrm{op}}
\newcommand{\verti}{\mathrm{vert}}
\newcommand{\Spin}{\operatorname{Spin}}
\newcommand{\ev}{\operatorname{ev}}
\newcommand{\dd}{\mathrm d}
\title[K-area bounds for smooth sphere bundles]{K-area bounds for smooth sphere bundles}
\author{Haoxuan Cheng}
\address{School of Mathematical Sciences, Fudan University, Shanghai 200433, China}
\email{hxcheng25@m.fudan.edu.cn}
\date{}

\begin{document}
\raggedbottom
\begin{abstract}
Let $E\to B$ be a smooth $S^n$ bundle with oriented Riemannian total
space $(E,g)$, where $n\ge2$. We prove
$\Karea(E,g)\le4\pi n(n-1)/\sigma$ whenever every fibre has intrinsic
scalar curvature at least $\sigma>0$, independently of the dimension
of $B$. In particular, unit round fibres give $\Karea(E,g)\le4\pi$.
For arbitrary metrics on two-sphere fibres, K-area is bounded by the
supremum of their areas. We also obtain stronger bounds from
fibrewise scalar spectra.
\end{abstract}
\maketitle

\section{Introduction and statements}\label{sec:intro}

K-area relates the curvature of vector bundles to the topology of the
underlying manifold. We bound the K-area of a sphere bundle's total
space in terms of the intrinsic geometry of its fibres.

Throughout the paper, manifolds are smooth, finite-dimensional, Hausdorff,
second countable, and without boundary, unless a boundary is explicitly
introduced in a construction. Let $\pi:E\to B$ be a nonempty smooth
locally trivial $S^n$ bundle, where $n\ge2$. Write $E_b=\pi^{-1}(b)$ and
$T_\pi E=\ker\dd\pi$. The total space $E$ is oriented and carries a smooth
Riemannian metric $g$. The base and the vertical tangent bundle need not
be oriented. The induced metric on $E_b$ is the restriction of $g$ to
$TE_b$; no submersion condition is imposed on $g$.

We first specify the K-area convention. For a finite-rank Hermitian complex
vector bundle $V\to E$ with unitary connection $\nabla$, its curvature is
the endomorphism-valued two-form
\[
 F_\nabla(X,Y)=\nabla_X\nabla_Y-\nabla_Y\nabla_X-\nabla_{[X,Y]}.
\]
Here $X,Y$ are vector fields on $E$, and the displayed operators act on
smooth sections of $V$.
Define its operator comass by
\begin{equation}\label{eq:comass}
 \|F_\nabla\|_g
 =\sup_{x\in E}\ 
   \sup_{\substack{u,v\in T_xE\\|u|_g=|v|_g=1,\ g(u,v)=0}}
       \|F_\nabla(u,v)\|_\op.
\end{equation}
The norm inside the supremum is the operator norm on $V_x$ defined by its
Hermitian metric. We use the Chern--Weil conventions
\[
 c(V,\nabla)=\det\left(I+\frac{i}{2\pi}F_\nabla\right),
 \qquad
 \ch(V,\nabla)=\operatorname{tr}\exp\left(\frac{i}{2\pi}F_\nabla\right).
\]
Write $\ch_j(V)$ for the component of real degree $2j$.
For background on characteristic classes and their differential-form
representatives, see \cite[\S\S14, 16 and Appendix~C]{MilnorStasheff1974}.

\begin{definition}\label{def:karea}
Suppose $\dim E=2m>0$. A pair $(V,\nabla)$ is admissible if at least one
number
\begin{equation}\label{eq:chern-number}
 \int_E c_{i_1}(V,\nabla)\wedge\cdots\wedge c_{i_\ell}(V,\nabla),
 \qquad i_j\ge1,\quad i_1+\cdots+i_\ell=m,
\end{equation}
is nonzero. If $E$ is noncompact, we also require a compact set $C\subset E$
and a global unitary frame $\tau$ on $E\setminus C$ in which $\nabla=\dd$.
The forms in \eqref{eq:chern-number} then have compact support. Set
\[
 \Karea(E,g)=\sup_{(V,\nabla)\ \mathrm{admissible}}
                 \frac{2\pi}{\|F_\nabla\|_g},
\]
with supremum zero if there are no admissible pairs. For odd-dimensional
$E$, set
\begin{equation}\label{eq:odd}
 \Karea(E,g)=\Karea(E\times\R,g+\dd t^2).
\end{equation}
\end{definition}

This is the normalization and real-line convention in
\cite[p.~31]{Gromov2017}. In the noncompact case we use the stated
parallel-frame meaning of trivial flatness at infinity, rather than the
weaker condition that curvature merely vanish outside a compact set.
The frame may be any frame satisfying that condition; it is not required
to extend over $E$. Formula~\eqref{eq:odd} uses all admissible bundles on
$E\times\R$, not only bundles pulled back from $E$.

For the spectral estimate, write
$h_b=g|_{E_b}$, let $\Delta_{h_b}=\dd^*\dd$ be the nonnegative
Laplacian on functions, and let $\operatorname{Sc}_{h_b}$ be the
intrinsic scalar curvature. Define
\begin{equation}\label{eq:scalar-spectrum}
 \lambda(b)=\lambda_{\min}
       \left(\Delta_{h_b}+\frac{\operatorname{Sc}_{h_b}}4\right).
\end{equation}
The lowest eigenvalue exists because each fibre is closed.

\begin{theorem}\label{thm:spectral}
Let $\pi:E\to B$ be a smooth $S^n$ bundle with oriented total space,
where $n\ge2$, and let $g$ be a smooth Riemannian metric on $E$.
If $\lambda(b)\ge\lambda_0>0$ for every $b\in B$, then
\begin{equation}\label{eq:spectral-bound}
 \Karea(E,g)\le\frac{\pi n(n-1)}{\lambda_0}.
\end{equation}
In particular, if $\operatorname{Sc}_{h_b}\ge\sigma>0$ on every
fibre, then
\begin{equation}\label{eq:scalar-bound}
 \Karea(E,g)\le\frac{4\pi n(n-1)}{\sigma}.
\end{equation}
\end{theorem}

\begin{corollary}\label{thm:round}
Let $\pi:E\to B$ be a smooth $S^n$ bundle with oriented total space,
where $n\ge2$, and let $g$ be a smooth Riemannian metric on $E$.
If each $(E_b,g|_{E_b})$ is isometric to the unit round
sphere, then
\begin{equation}\label{eq:round-bound}
 \Karea(E,g)\le4\pi.
\end{equation}
\end{corollary}

\begin{proof}
On a unit round $S^n$, scalar curvature is $n(n-1)$ and constant
functions minimize $\Delta+\operatorname{Sc}/4$. Thus
$\lambda(b)=n(n-1)/4$ for every fibre. Substitution in
Theorem~\ref{thm:spectral} gives $\Karea(E,g)\le4\pi$.
\end{proof}

For two-sphere fibres the area gives a bound without a curvature assumption.

\begin{theorem}\label{thm:area}
Let $\pi:E\to B$ be a smooth $S^2$ bundle with oriented total space, and
let $g$ be any smooth Riemannian metric on $E$. Then
\begin{equation}\label{eq:area-bound}
 \Karea(E,g)\le A,
 \qquad A=\sup_{b\in B}\operatorname{Area}_g(E_b).
\end{equation}
The right-hand side is allowed to be $+\infty$.
\end{theorem}

Corollary~\ref{thm:round} and Theorem~\ref{thm:area} answer Question~25 in
Gromov's 2017 manuscript \emph{101 Questions, Problems and Conjectures
around Scalar Curvature} \cite[Question~25(iii), p.~36]{Gromov2017}.
There he asks whether the $S^n$- and $S^2$-product inequalities on
pages~35--36 extend to nontrivial sphere fibrations.

All three statements use Definition~\ref{def:karea}, including its
noncompact and odd-dimensional conventions. None assumes completeness,
a section of $\pi$, a spin structure, a linear structure group,
or totally geodesic fibres. In particular, the base need not be a sphere.
In dimension two the independent area bound remains stronger:
if $\operatorname{Sc}_{h_b}\ge\sigma$, Gauss--Bonnet gives
$\operatorname{Area}(E_b)\le8\pi/\sigma$, whereas the area theorem
itself needs no positive-curvature or spectral assumption.
For $n\ge3$, Theorem~\ref{thm:conformal} replaces $\lambda(b)$ by
the lowest eigenvalue of one quarter of the conformal Laplacian.
This quantity is at least $\lambda(b)$, so that refinement improves
the spectral bound while retaining its constant.

The relation between small-curvature bundles and scalar curvature
underlies Gromov's development of K-area \cite{Gromov1996}.
Gromov and Lawson \cite{GromovLawson1983} use twisted Dirac operators
and almost flat bundles to obstruct positive scalar curvature;
background on spinor bundles, Dirac operators, and index theory can be
found in Lawson and Michelsohn
\cite[Chapters~II--III]{LawsonMichelsohn1989}.
For products, Listing's estimate
\cite[Proposition~1.3.2]{Listing2012} applies to a closed spin manifold
of positive scalar curvature and a compact parameter manifold.
For the unit $S^n$, its stabilized Chern-character K-area constant is
$2$ in the normalization $1/\|F\|$, or $4\pi$ in ours.
Thus both the product constant and the family-Dirac method are already
present in that result. Listing's parametric descent theorem
\cite[Theorem~2.6]{Listing2013} also treats small vertical curvature
on products. Passing to a general sphere bundle requires the
fibrewise constructions to respect its transition maps.

A further issue is the use of all Chern numbers in
Definition~\ref{def:karea}. Wang's comparison with $\widehat A$-cowaist
\cite[Theorem~1.1]{Wang2023} and the Adams-operation comparison of
B\"ar and Hanke \cite[Lemma~5 and Theorem~6]{BaerHanke2024} introduce
a factor depending on the total dimension.
Related estimates for foliations also use different detection or
geometric hypotheses. Su and Wang \cite{SuWang2025} work with
complete foliated manifolds, covering $\widehat A$-cowaist, and spin
hypotheses. Benameur and Heitsch \cite[\S5]{BenameurHeitsch2019}
study a covering spin-foliation K-area under a Hausdorff
homotopy-groupoid hypothesis, using a restricted class of detecting
bundles and the curvature term in the Dirac square.
These invariants differ from Definition~\ref{def:karea}.
Gromov's later discussion of sphere fibrations and K-cowaist
\cite[\S\S6.3.2, 6.3.4--6.3.5]{Gromov2021} includes constructions
of detecting bundles from base and fibre data. Our estimates give a curvature
lower bound for every detecting bundle for the given total metric.

The spectral argument uses the family index theorem
\cite{AtiyahSinger1971,Bismut1986} and Liu's eta-form formula
\cite{Liu2017}. Over a closed base with oriented fibres, applying the
formula to both the untwisted and twisted Dirac families gives
vanishing fibre integrals of Chern characters.
The sphere Gysin sequence then makes all Chern classes pullbacks
from the base. This avoids the dimension factor arising from
operations on the detecting bundle.

A projective base change supplies the required vertical $\Spin^c$
structure; injectivity of pullback returns the vanishing to the
original base. Auxiliary fibrations for removing spin assumptions
occur in \cite[\S5$\frac34$]{Gromov1996}, and our base change uses
\cite[Proposition~2.1]{AtiyahSegal2005}. The scalar and conformal
spectral estimates are based on the methods discussed in
\cite[\S6.1.2]{Gromov2021}. For noncompact total spaces, we preserve
the detecting integrals by doubling a compact restriction of the bundle.

\section{Two-sphere fibres and holomorphic descent}\label{sec:two}

For oriented two-sphere fibres, small curvature makes the detecting
bundle descend to the base through its fibrewise holomorphic sections.
We first prove the estimate on a single fibre.

\begin{lemma}\label{lem:holomorphic}
Let $(\Sigma,h)$ be an oriented Riemannian two-sphere and let $V\to\Sigma$
be a Hermitian vector bundle with unitary connection $\nabla$. Give
$\Sigma$ the complex structure determined by $h$ and its orientation.
If
\begin{equation}\label{eq:fibre-threshold}
 \int_\Sigma \|F_\nabla(e_1,e_2)\|_\op\,\dd A_h<2\pi,
\end{equation}
where $(e_1,e_2)$ is any local positively oriented orthonormal frame,
then the holomorphic bundle $(V,\bar\partial_\nabla)$ is trivial.
\end{lemma}

\begin{proof}
The operator $\bar\partial_\nabla$ is integrable because a complex curve
has no forms of type $(0,2)$. Its Chern connection for the given
Hermitian metric is $\nabla$. The complex curve $\Sigma$ has genus zero,
so it is biholomorphic to $\C P^1$. By Grothendieck's splitting theorem
\cite[Theorem~1.1]{Grothendieck1957}, the holomorphic bundle is a sum
$\bigoplus_j\mathcal O(k_j)$, with $k_j\in\Z$.

For a holomorphic line subbundle $L\subset V$, use the smooth orthogonal
decomposition $V=L\oplus L^\perp$. The connection and its line curvature
have the form
\[
 \nabla=
 \begin{pmatrix}\nabla^L&-\beta^*\\ \beta&\nabla^Q\end{pmatrix},
 \qquad
 F_L=P_LF_\nabla|_L+\beta^*\wedge\beta,
\]
where $P_L$ is orthogonal projection and $\beta$ has type $(1,0)$.
In a positively oriented complex coordinate $z=x+iy$, write
$\beta=B\,\dd z$. Then
\[
 i\beta^*\wedge\beta=-2B^*B\,\dd x\wedge\dd y.
\]
The degree formula therefore gives
\begin{equation}\label{eq:line-degree}
 \deg L=\frac{1}{2\pi}\int_\Sigma iF_L
 \le\frac{1}{2\pi}\int_\Sigma
                    \|F_\nabla(e_1,e_2)\|_\op\,\dd A_h<1.
\end{equation}
Each positive $k_j$ would contradict \eqref{eq:line-degree}. Applying
the same argument to $V^*$ rules out negative $k_j$, since the dual
curvature has the same operator norm. Hence every $k_j$ is zero.
\end{proof}

The spaces of holomorphic sections give compatible trivializations
as the fibre varies.

\begin{lemma}\label{lem:descent}
Suppose the vertical tangent bundle of $\pi:E\to B$ is oriented and the
fibres are two-spheres. Let $V\to E$ have rank $r$ and unitary connection
$\nabla$. If \eqref{eq:fibre-threshold} holds on every fibre, there is a
smooth rank-$r$ complex bundle $H\to B$ and a natural smooth isomorphism
\[
 \ev:\pi^*H\longrightarrow V.
\]
If $\nabla=\dd_\tau$ outside a compact set $C\subset E$, there is a
frame $\sigma$ of $H$ over $B\setminus\pi(C)$ such that
\begin{equation}\label{eq:frames}
 \tau\circ\ev=\pi^*\sigma
 \quad\text{on }\pi^{-1}(B\setminus\pi(C)).
\end{equation}
\end{lemma}

\begin{proof}
The vertical orientation and metric define a smooth vertical complex
structure. For $b\in B$, let
\[
 D_b=\bar\partial_{\nabla|E_b},\qquad H_b=\ker D_b.
\]
Lemma~\ref{lem:holomorphic} implies $\dim_\C H_b=r$ and
$\operatorname{coker}D_b=H^1(\C P^1,\mathcal O^r)=0$.
Evaluation at any point of $E_b$ is an isomorphism from $H_b$ to the
corresponding fibre of $V$.

Fix $b_0$ and identify the fibres
and the relevant bundles over a small coordinate neighbourhood of
$b_0$, using a local trivialization and an auxiliary connection. The
operators $D_b$ become a smooth family between fixed Sobolev spaces
$H^{s+1}$ and $H^s$, with $s>2$. Choose an elliptic right inverse $Q_0$
for $D_{b_0}$. Near $b_0$, the operators
\[
 Q_b=Q_0(D_bQ_0)^{-1},\qquad P_b=I-Q_bD_b
\]
are smooth. Indeed, $D_bQ_0$ is close to the identity in operator norm,
so its inverse is given locally by a convergent Neumann series. Direct
composition gives
\[
 D_bQ_b=I,\qquad D_bP_b=0,\qquad
 P_b^2=P_b,\qquad P_b|_{\ker D_b}=I.
\]
Thus $P_b$ projects onto $\ker D_b$. Projecting a basis of
$\ker D_{b_0}$ gives a local basis of $H_b$. The right inverse may be
chosen as $D_{b_0}^*(D_{b_0}D_{b_0}^*)^{-1}$, so it is compatible with
all Sobolev orders. Parameter-dependent elliptic regularity then makes
the local basis smooth jointly in the base and fibre variables.

The resulting kernel bundle is intrinsic. To check the gluing, denote
the coordinate changes on sections and on $(0,1)$-forms by
$T^0_{ij}$ and $T^1_{ij}$, respectively, on an overlap of two local
charts. If $D_i$ and $D_j$ are the local operator families, then
\[
 D_iT^0_{ij}=T^1_{ij}D_j,\qquad
 h_{ij}:=T^0_{ij}|_{\ker D_j}:\ker D_j\longrightarrow\ker D_i.
\]
The original coordinate changes obey the cocycle identity, so their
restrictions satisfy $h_{ij}h_{jk}=h_{ik}$. If $u_{ij}$ is the
corresponding coordinate change on $V$, evaluation obeys
\[
 u_{ij}\ev_j=\ev_i\,\pi^*h_{ij}.
\]
It therefore
defines a smooth global map $\pi^*H\to V$, which is invertible at every
point. No uniform spectral estimate on a noncompact base is needed for
this local construction.

Put $D=\pi(C)$. For $b\notin D$, the entire fibre lies in the domain of
$\tau$. Its holomorphic sections are exactly
$s_z(x)=\tau_x^{-1}z$, with $z\in\C^r$, because holomorphic functions
on a compact connected Riemann surface are constant. Define
$\sigma_b(s_z)=z$. These sections vary smoothly with $b$, and
$\tau_x(\ev_x(s_z))=z$, which proves \eqref{eq:frames}.
\end{proof}

The projection of a smooth bundle with compact fibre is proper. Indeed,
a compact subset of the base is covered by finitely many neighbourhoods
whose compact closures lie in trivializing charts; the inverse images
of those closures are compact products. We will use this fact to preserve
compact support when enlarging a set to a union of fibres.

\begin{proposition}\label{prop:two-vanish}
Suppose $\dim E=2m$ and $E$ is oriented. Let $(V,\nabla)$ be a Hermitian
bundle with unitary connection satisfying
\eqref{eq:fibre-threshold} on every two-sphere fibre. In the noncompact
case assume the parallel-frame condition in
Definition~\ref{def:karea}. Then all numbers in
\eqref{eq:chern-number} vanish.
\end{proposition}

\begin{proof}
First suppose the vertical tangent bundle is oriented. Apply
Lemma~\ref{lem:descent}. If $E$ is compact, then $B$ is compact and
$V\cong\pi^*H$ makes every top-degree Chern monomial a pullback of a
class of degree $2m>\dim B=2m-2$, hence zero.

For noncompact $E$, write $D=\pi(C)$ and choose a compact neighbourhood
$D'$ of $D$ in $B$. Choose a complex connection $\nabla^H$ on $H$ that
equals $\dd_\sigma$ outside $D'$. Explicitly, start with a connection
$\nabla^{\mathrm{aux}}$ and a smooth function $\chi$ equal to $1$ near
$D$ and supported in $\operatorname{int}D'$. On $B\setminus D$, take
$\nabla^H=\chi\nabla^{\mathrm{aux}}+(1-\chi)\dd_\sigma$; near $D$
this extends as $\nabla^{\mathrm{aux}}$.
Transport $\pi^*\nabla^H$ to a connection
$\nabla^0$ on $V$ through evaluation. The connection $\nabla^0$ need
not be unitary for the original metric on $V$; it is used only in
Chern--Weil theory. By \eqref{eq:frames},
$a=\nabla-\nabla^0$ is supported in the compact set $\pi^{-1}(D')$.

Let $P$ be the invariant polynomial of degree $m$ corresponding to a
top-degree Chern monomial, with the factors $i/(2\pi)$ included in its
definition. Set $\nabla^t=\nabla^0+ta$ and $F_t=(\nabla^t)^2$. If
$\widetilde P$ is the symmetric polarization of $P$, the transgression
formula is
\begin{equation}\label{eq:transgression}
 P(F_\nabla)-P(F_{\nabla^0})
 =\dd\left(m\int_0^1\widetilde P(a,F_t,\ldots,F_t)\,\dd t\right).
\end{equation}
For completeness, the differentiation underlying this formula is
\begin{align*}
 \frac{\partial}{\partial t}P(F_t)
 &=m\widetilde P(\dd_{\nabla^t}a,F_t,\ldots,F_t)\\
 &=m\,\dd\widetilde P(a,F_t,\ldots,F_t).
\end{align*}
The first equality uses $\partial_tF_t=\dd_{\nabla^t}a$; the second
uses invariance of $\widetilde P$ and
$\dd_{\nabla^t}F_t=0$. Integrating in $t$ gives
\eqref{eq:transgression}.
The primitive on the right has compact support. On the other hand,
$P(F_{\nabla^0})=\pi^*P(F_{\nabla^H})=0$ by the dimension of $B$.
Compactly supported Stokes' theorem proves the required integral
vanishing. 

If the vertical tangent bundle is not oriented, take the two-sheeted
cover $\widetilde B\to B$ that records an orientation of the fibres,
and pull back all data to $q:\widetilde E\to E$. This finite covering
preserves the parallel-frame condition and the fibrewise estimate.
Give $\widetilde E$ the pulled-back total orientation. The preceding
argument applies there, and
\[
 0=\int_{\widetilde E}q^*P(F_\nabla)
   =2\int_E P(F_\nabla).
\]
Thus every required Chern number vanishes also on $E$.
\end{proof}

\begin{proof}[Proof of Theorem~\ref{thm:area}]
If $A=+\infty$, the assertion is the extended-real inequality. Otherwise
$0<A<\infty$. First suppose $E$ is even-dimensional. If an admissible
pair had $K:=\|F_\nabla\|_g<2\pi/A$, then on each fibre
\[
 \int_{E_b}\|F_\nabla(e_1,e_2)\|_\op\,\dd A
 \le K\operatorname{Area}_g(E_b)\le KA<2\pi.
\]
Proposition~\ref{prop:two-vanish} would contradict admissibility.
Consequently every admissible pair has $K\ge2\pi/A$, proving the bound.
If $E$ is odd-dimensional, apply this same even-dimensional argument to
$E\times\R\to B\times\R$ with the product metric. The fibre areas are
unchanged, and the argument applies to every admissible bundle on
$E\times\R$. Definition~\ref{def:karea} completes the proof.
\end{proof}

\section{Family indices and cohomological descent}\label{sec:round}

We first consider a closed base. For oriented fibres, write $\pi_!$
for fibre integration. The following lemma reduces Chern-number
vanishing to vanishing of these fibre integrals.

For a smooth vertical metric $h$ on $T_\pi E$, let
$\kappa=\|F_\nabla\|_\verti$ denote the restriction of
\eqref{eq:comass} to orthonormal vertical vectors. If $h=g|_{T_\pi E}$,
then $\kappa\le\|F_\nabla\|_g$.

\begin{lemma}\label{lem:gysin-chern}
Let $\pi:E\to B$ be a smooth $S^n$ bundle with oriented vertical tangent
bundle, where $n\ge2$. If a complex bundle $V\to E$ satisfies
$\pi_!\ch_j(V)=0$ for every $j$ in real cohomology, then each $c_i(V)$
is pulled back from $B$. Every Chern monomial of real degree greater
than $\dim B$ consequently vanishes.
\end{lemma}

\begin{proof}
The Gysin sequence of an oriented sphere bundle is exact at its middle
term:
\begin{equation}\label{eq:gysin}
 H^k(B;\R)\xrightarrow{\pi^*}H^k(E;\R)
 \xrightarrow{\pi_!}H^{k-n}(B;\R).
\end{equation}
Here $\pi_!$ is integration along the oriented fibres.
This sequence does not require a linear structure group
\cite[\S4.D]{Hatcher2002}. To describe the topological construction,
extend each transition map of $S^n$ radially to its cone. This gives a
topological disk bundle $D\to B$ with boundary $E$, and $D$ retracts
onto $B$. The exact sequence of $(D,E)$ and the oriented Thom
isomorphism give \eqref{eq:gysin}. The radial extensions need not be
smooth at the cone points: only their topological Thom class is used.
The connecting map followed by the inverse Thom isomorphism agrees
with smooth fibre integration, with compatible orientations.

Thus $\ch_j(V)=\pi^*a_j$ for some $a_j\in H^{2j}(B;\R)$.
The Newton identity
\begin{equation}\label{eq:newton}
 \sum_{j\ge0}c_j(V)t^j
 =\exp\left(\sum_{j\ge1}(-1)^{j-1}(j-1)!\ch_j(V)t^j\right)
\end{equation}
expresses every Chern class as a polynomial in the Chern characters.
For formal Chern roots $x_\alpha$, this follows by expanding
$\log\prod_\alpha(1+x_\alpha t)$ and using
$\ch_j(V)=\sum_\alpha x_\alpha^j/j!$. Each $c_i(V)$ is therefore
$\pi^*b_i$ for some $b_i\in H^{2i}(B;\R)$, and
\[
 \prod_{\nu=1}^{\ell}c_{i_\nu}(V)
 =\pi^*\left(\prod_{\nu=1}^{\ell}b_{i_\nu}\right)=0
 \quad\text{if }2\sum_{\nu=1}^{\ell}i_\nu>\dim B.
\]
\end{proof}

A determinant connection pulled back from the base has zero vertical
curvature. Under this condition, the family index formula gives the
following vanishing statement.

\begin{lemma}\label{lem:two-indices}
Let $B$ be closed and let $\pi:E\to B$ be a smooth $S^n$ bundle,
$n\ge2$, with a vertical $\Spin^c$ structure whose determinant line is
$L=\pi^*L_0$ for a complex line bundle $L_0\to B$.
Give $L$ a connection pulled back from $L_0$.
If the untwisted vertical Dirac family and its twist by a Hermitian
bundle $V$ are both invertible, then $\pi_!\ch_j(V)=0$ for every $j$.
\end{lemma}

\begin{proof}
Choose a horizontal distribution and the associated vertical tangent
connection. The characteristic class appearing in the family index
formula is
\[
 \mathcal A=\widehat A(T_\pi E)e^{c_1(L)/2}\in H^{\mathrm{even}}(E;\R).
\]
For a twisting bundle $V$, write
\[
 \Omega_V=\int_{E/B}\widehat A(T_\pi E,\nabla^{T_\pi E})
                    \exp\left(\frac{iF_L}{4\pi}\right)\ch(V,\nabla).
\]
Liu's eta-form formula \cite[(1.3), (2.57), (2.81)]{Liu2017},
with the acting group trivial, gives
\begin{equation}\label{eq:eta-interface}
 \dd\widetilde\eta_V=
 \begin{cases}
 \Omega_V-\ch(\ker D_V^+,\nabla^{\ker,+})
          +\ch(\ker D_V^-,\nabla^{\ker,-}),&n\text{ even},\\
 \Omega_V,&n\text{ odd}.
 \end{cases}
\end{equation}
The hypotheses are a smooth geometric family with closed oriented
$\Spin^c$ fibres, a Hermitian twisting bundle, and locally constant
kernel dimension. Invertibility makes every kernel zero, including
both half-spinor kernels when $n$ is even. Thus
\begin{equation}\label{eq:family-zero}
 \pi_!(\mathcal A\ch(V))=0\quad\text{in }H^*(B;\R).
\end{equation}
For odd $n$ this is a statement in every odd base degree, obtained
from the eta form; it is not a numerical index calculation on a fibre.

Apply the same formula first to the untwisted family. It gives
$\pi_!\mathcal A=0$. By \eqref{eq:gysin}, degree by degree,
\[
 \mathcal A=\pi^*a,\qquad a\in H^{\mathrm{even}}(B;\R),\qquad a_0=1.
\]
The choice $a_0=1$ is possible because the fibres are connected.
The positive-degree class $u=a-1$ is nilpotent on the finite-dimensional
base. For sufficiently large $N$,
$a^{-1}=\sum_{q=0}^{N-1}(-u)^q$.
The twisted formula \eqref{eq:family-zero} and the projection formula
now give
\[
 0=\pi_!(\pi^*a\ch(V))=a\,\pi_!\ch(V),\qquad
 \pi_!\ch(V)=0.
\]
Separating homogeneous degrees proves the assertion. This second use
of the family formula replaces any assumption that $T_\pi E$ is
stably pulled back from a vector bundle on $B$.
\end{proof}

The determinant condition in Lemma~\ref{lem:two-indices} can be
arranged after a base change with injective pullback in real cohomology.

\begin{lemma}\label{lem:brauer}
Let $B$ be closed and let $\pi:E\to B$ be a smooth $S^n$ bundle,
$n\ge2$, with oriented vertical tangent bundle. There is a closed
smooth manifold $\mathcal F$ and a smooth map $\rho:\mathcal F\to B$
such that $\rho^*:H^*(B;\R)\to H^*(\mathcal F;\R)$ is injective,
and the pulled-back sphere bundle $\pi_{\mathcal F}:E_{\mathcal F}\to
\mathcal F$ has a vertical $\Spin^c$ structure with determinant
$\pi_{\mathcal F}^*L_0$ for a line bundle $L_0\to\mathcal F$.
\end{lemma}

\begin{proof}
The mod-two Gysin sequence in degree two contains
\[
 0\longrightarrow H^2(B;\Z/2)\xrightarrow{\pi^*}
 H^2(E;\Z/2)\xrightarrow{\pi_!}H^{2-n}(B;\Z/2).
\]
For $n\ge3$ the last group is zero. For $n=2$,
$\pi_!w_2(T_\pi E)$ has value
$\langle w_2(TS^2),[S^2]\rangle=\chi(S^2)\bmod2=0$ on every
component. In either case there is a unique $a_2\in H^2(B;\Z/2)$
with $w_2(T_\pi E)=\pi^*a_2$.
Let $\alpha=\beta(a_2)\in H^3(B;\Z)$, where $\beta$ is the
Bockstein for $0\to\Z\xrightarrow{2}\Z\to\Z/2\to0$.
This is a torsion class.

A closed smooth manifold has finite CW type. The finite-dimensional
realization of torsion classes in $H^3$ therefore gives a principal
$\operatorname{PU}(r)$ bundle $Q\to B$ whose obstruction to a lift
through $U(r)\to\operatorname{PU}(r)$ is $\alpha$
\cite[Proposition~2.1(i), (iv)--(vi)]{AtiyahSegal2005}.
The integer $r$ need not be the order of $\alpha$; the distinction
between the order of a Brauer class and its possible realization degrees
is the topological period-index problem \cite[\S1]{AntieauWilliams2014}.
We take a smooth representative: for a compact Lie group, a
finite-dimensional smooth approximation to its classifying space,
followed by smooth approximation of the classifying map from $B$,
provides such a representative without changing the obstruction.
Set
\[
 \mathcal F=Q\times_{\operatorname{PU}(r)}\C P^{r-1},
 \qquad \rho:\mathcal F\longrightarrow B.
\]
If $\alpha=0$, we may instead take $r=1$ and $\rho$ to be the identity.

The stabilizer of a
point of $\C P^{r-1}$ is
$H=P(U(1)\times U(r-1))$. Its inclusion into
$\operatorname{PU}(r)$ lifts to $U(r)$ by the homomorphism
\[
 [\operatorname{diag}(z,A)]\longmapsto
                 \operatorname{diag}(1,z^{-1}A).
\]
Multiplying $(z,A)$ by a common scalar leaves this expression unchanged.
The principal $H$ bundle $Q\to Q/H=\mathcal F$, extended by this
homomorphism, lifts $\rho^*Q$. Consequently $\rho^*\alpha=0$ in
integer cohomology. On the other hand, the vertical tangent bundle
$T_\rho\mathcal F$ is complex and has its complex orientation.
Its Euler class satisfies
\begin{equation}\label{eq:projective-transfer}
 \rho_!\bigl(e(T_\rho\mathcal F)\smile\rho^*z\bigr)=r z,
 \qquad z\in H^*(B;\R),
\end{equation}
since $\chi(\C P^{r-1})=r$. This proves real injectivity, which is
compatible with killing an integer torsion class.

The Bockstein exact sequence now gives an integer lift
$c\in H^2(\mathcal F;\Z)$ of $\rho^*a_2$.
Choose a line bundle $L_0\to\mathcal F$ with $c_1(L_0)=c$.
For the pullback square
\begin{equation}\label{eq:base-square}
 \begin{array}{ccc}
 E_{\mathcal F}=E\times_B\mathcal F&\xrightarrow{\ q\ }&E\\
 {\scriptstyle\pi_{\mathcal F}}\downarrow&&\downarrow{\scriptstyle\pi}\\
 \mathcal F&\xrightarrow{\ \rho\ }&B,
 \end{array}
\end{equation}
we have
$w_2(T_{\pi_{\mathcal F}}E_{\mathcal F})
=c_1(\pi_{\mathcal F}^*L_0)\bmod2$.
The $\Spin^c$ lifting criterion with specified determinant
\cite[\S1.1]{BotvinnikRosenberg2024} supplies the required structure.
Equivalently, choose a spin structure on the oriented direct sum
$T_{\pi_{\mathcal F}}E_{\mathcal F}\oplus
(\pi_{\mathcal F}^*L_0)_{\R}$ and use the reduction to its two summands.
Choosing a connection on $L_0$ and pulling it back makes the determinant
curvature zero on vertical two-planes. In dimension two this choice
is essential; an arbitrary $\Spin^c$ structure on $S^2$ can have
nontrivial determinant along the fibre.
\end{proof}

The fibrewise Dirac estimate now gives the required curvature threshold.

\begin{proposition}\label{prop:closed-spectral}
Let $B$ be closed and let $\pi:E\to B$ be a smooth $S^n$ bundle,
$n\ge2$, with a smooth vertical metric $h$. Suppose that on every fibre
\[
 \lambda_{\min}(\Delta_h+\operatorname{Sc}_h/4)\ge\lambda_0>0.
\]
If a Hermitian bundle $V\to E$ with unitary connection satisfies
$\kappa<2\lambda_0/[n(n-1)]$, then every Chern monomial of real degree
greater than $\dim B$ vanishes. If $T_\pi E$ is oriented, each
$\ch_j(V)$ and $c_j(V)$ is pulled back from $B$ in real cohomology.
\end{proposition}

\begin{proof}
First assume that $T_\pi E$ is oriented and has a $\Spin^c$ structure
with determinant $L=\pi^*L_0$. Give $L$ a pulled-back connection.
For the vertical spinor bundle $\mathcal S$, the fibrewise twisted
Dirac operator is
\[
 D_{V,b}s=\sum_{i=1}^n c(e_i)\nabla^{\mathcal S\otimes V}_{e_i}s,
\]
where $(e_i)$ is an orthonormal vertical frame and $c$ is Clifford
multiplication with $c(e_i)^2=-1$.
The square formula, in the $i<j$ convention
\cite[pp.~6--7]{Melrose2006}, is
\begin{equation}\label{eq:dirac-square}
 D_{V,b}^2=\nabla^*\nabla+\frac{\operatorname{Sc}_{h_b}}4
       +\sum_{i<j}c(e_i)c(e_j)\otimes F_\nabla(e_i,e_j).
\end{equation}
The determinant curvature term is zero. Each twisting summand is
self-adjoint and has operator norm at most $\kappa$, so their sum
has norm at most $n(n-1)\kappa/2$, independently of $\rk V$.

Kato's inequality $|\dd|s||\le|\nabla s|$ and the scalar Rayleigh
quotient give
\begin{align}
 \|D_{V,b}s\|_{L^2}^2
 &\ge\int_{E_b}\left(|\dd|s||^2+
       \frac{\operatorname{Sc}_{h_b}}4|s|^2\right)
       -\frac{n(n-1)}2\kappa\|s\|_{L^2}^2\notag\\
 &\ge\left(\lambda_0-\frac{n(n-1)}2\kappa\right)
                         \|s\|_{L^2}^2>0
 \quad\text{for }s\ne0.\label{eq:gap}
\end{align}
The first integral uses the fibre volume form. At zeros of $s$, one
may use $(|s|^2+\varepsilon^2)^{1/2}$ and pass to the $H^1$ limit;
the scalar potential is bounded on each closed fibre.
Thus both the twisted and the untwisted families are invertible.
Lemmas~\ref{lem:two-indices} and~\ref{lem:gysin-chern} prove the
assertion in this case.

For an arbitrary oriented vertical tangent bundle, use
Lemma~\ref{lem:brauer} and the square \eqref{eq:base-square}.
Give each sphere fibre upstairs the pullback of its original metric.
The map $q$ is then an isometry on that fibre. The spectral lower bound
and the vertical curvature norm of $q^*V$ are unchanged. No pullback
of the total metric is required; an auxiliary horizontal distribution
can be chosen independently.
The preceding case gives $(\pi_{\mathcal F})_!\ch_j(q^*V)=0$.
Naturality of oriented fibre integration gives
\begin{equation}\label{eq:brauer-return}
 \rho^*(\pi_!\ch_j(V))=(\pi_{\mathcal F})_!q^*\ch_j(V)=0.
\end{equation}
By \eqref{eq:projective-transfer}, $\pi_!\ch_j(V)=0$ on the original
base $B$. Lemma~\ref{lem:gysin-chern} now uses $\dim B$, not
$\dim\mathcal F$. The auxiliary rank $r$ consequently does not enter
the curvature constant.

Finally, the fibre orientation local system is a local system on
$B$. Pull back to its orientation double cover. The oriented case
makes every Chern monomial of degree greater than $\dim B$ vanish on
the pulled-back total space. For a finite covering, transfer composed
with pullback is multiplication by the number of sheets, so pullback
is injective over $\R$. The same monomials vanish on $E$.
\end{proof}

\section{Compact reduction preserving the fibre metrics}\label{sec:noncompact}

Ordinary cohomological vanishing does not imply vanishing of compactly
supported Chern--Weil integrals. The following double preserves these
integrals and the fibre metrics, reducing the problem to a closed base.

\begin{lemma}\label{lem:double}
Let $\pi:E\to B$ be a smooth $S^n$ bundle with oriented even-dimensional
total space, and let $g$ be a Riemannian metric on $E$.
Suppose $V\to E$ has a unitary connection $\nabla$ and a global unitary
parallel frame $\tau$ outside a compact set $C$.
There is a closed sphere bundle $\widehat\pi:\widehat E\to\widehat B$
with oriented total space, a metric $\widehat g$, and a Hermitian
bundle with connection $(\widehat V,\widehat\nabla)$ such that every
fibre of $\widehat E$ is isometric to a fibre of $E$ and
\begin{align}
 \|F_{\widehat\nabla}\|_\verti&=\|F_\nabla\|_\verti,
 &\|F_{\widehat\nabla}\|_{\widehat g}&=\|F_\nabla\|_g,
 \label{eq:double-norms}\\
 \int_{\widehat E}P(F_{\widehat\nabla})
   &=\int_E P(F_\nabla)\label{eq:double-integral}
\end{align}
for every top-degree Chern monomial $P$.
\end{lemma}

\begin{proof}
Choose a compact smooth codimension-zero domain $W\subset B$ with
$\pi(C)\subset\operatorname{int}W$ and a boundary collar disjoint
from $\pi(C)$. Such a domain is a suitable regular sublevel set of
a smooth proper exhaustion; whole closed components may be included.
A horizontal distribution lifts the normal coordinate vector field
of the collar. Its flow yields a smooth bundle identification
\[
 \Phi:E|_{\partial W}\times[0,\varepsilon)\longrightarrow
                         E|_{\mathrm{collar}}.
\]
The flow exists uniformly on a smaller collar because the fibre and
$\partial W$ are compact. 

For $y\in\partial W$, let $k_{y,t}=\Phi_t^*h_{(y,t)}$ be the induced
fibre metric on $E_y$. Choose a smooth function $\varphi$ on the collar
interval that equals zero near $0$, equals $t$ near the inner edge,
and satisfies $0\le\varphi(t)\le t$. Replace the vertical metric on
the collar by $h'$ determined by
\begin{equation}\label{eq:collar-metric}
 \Phi_t^*h'_{(y,t)}=k_{y,\varphi(t)}.
\end{equation}
Every new fibre is isometric to the original fibre over
$(y,\varphi(t))$. In particular, any uniform intrinsic scalar or
spectral bound is unchanged. Near the boundary this metric is constant
in $t$ in the coordinates supplied by $\Phi$, and near the inner
edge it agrees with the original metric.

Double $W$ to obtain $\widehat B$ and glue two copies of $E|_W$ using
$\Phi$ along the boundary bundle. The product collar charts give a
smooth sphere bundle $\widehat E\to\widehat B$. Reflect $h'$ to the
second half; it is smooth at the seam by \eqref{eq:collar-metric}.
Give the second half the opposite total orientation. This uses only
the orientation of $E$, not separate orientations of $B$ and $T_\pi E$.

On the positive half put the original $(V,\nabla)$. On the negative
half put a trivial bundle of the same rank with its trivial connection.
Use the specified frame $\tau$ to glue them. In this frame both
connection matrices are zero throughout a collar, so the connection
is smooth and creates no curvature at the seam. The frame may have
nontrivial homotopy class and need not extend to the interior of either
half. All nonzero curvature stays in the positive-half set $C$, where
the vertical metric is unchanged. Hence the vertical norm is preserved.

Choose any total metric $\widehat g_0$ realizing the doubled vertical
metric. On a neighbourhood of $C$ in the unchanged part of the
positive half, replace it by $g$ using a smooth convex interpolation
supported in that part. Both total metrics have the same vertical
restriction there, so this operation preserves the vertical metric.
The resulting $\widehat g$ agrees with $g$ wherever curvature can be
nonzero and proves \eqref{eq:double-norms}. 

Every positive-degree Chern--Weil monomial vanishes on the negative
half and near the seam, while on the positive half it is the original
form with the original orientation. Therefore
\[
 \int_{\widehat E}P(F_{\widehat\nabla})
 =\int_{E|_W}P(F_\nabla)=\int_E P(F_\nabla).
\]
Using a second copy of the original connection would instead cancel
these integrals, which is why the negative half was chosen flat.
If a component of $W$ is closed, its double consists of two disjoint
copies with the same choices of connections. Only finitely many
components meet $C$. If $C$ is empty, all detecting forms already
vanish and the same construction can be made over any compact domain.
\end{proof}

\begin{proof}[Proof of Theorem~\ref{thm:spectral}]
Suppose first that $E$ is even-dimensional and take an admissible
$(V,\nabla)$. If
\[
 \|F_\nabla\|_g<\frac{2\lambda_0}{n(n-1)},
\]
then the same strict inequality holds for its vertical norm.
If $E$ is compact, Proposition~\ref{prop:closed-spectral} makes every
top-degree Chern monomial zero, since $\dim E=\dim B+n>\dim B$.
If $E$ is noncompact, Lemma~\ref{lem:double} produces a closed sphere
bundle with the same lower spectral bound, curvature norms, and every
detecting integral. Apply the proposition on this double and use
\eqref{eq:double-integral}. In both cases all detecting numbers
vanish, contradicting admissibility.
Thus every admissible connection has norm at least
$2\lambda_0/[n(n-1)]$, and the $2\pi$ normalization gives
\eqref{eq:spectral-bound}.

For odd-dimensional $E$, apply this argument to
$E\times\R\to B\times\R$ with total metric $g+\dd t^2$.
The fibre metrics and their spectra are unchanged. Each admissible
bundle on $E\times\R$ has its own compact exceptional set and frame;
the double is constructed for that bundle on a compact domain in
$B\times\R$. No assumption that it is pulled back from $E$ is made.
Definition~\ref{def:karea} gives the same bound.
Finally $\operatorname{Sc}_{h_b}\ge\sigma$ implies
$\lambda(b)\ge\sigma/4$ by the Rayleigh quotient, proving
\eqref{eq:scalar-bound}.
\end{proof}

\section{A conformal spectral strengthening}\label{sec:conformal}

For $n\ge3$, the conformal Laplacian gives a stronger estimate.
For a metric $h$ on a closed $n$-manifold, define
\begin{equation}\label{eq:conformal-operator}
 Q_h=\frac{n-1}{n-2}\Delta_h+\frac{\operatorname{Sc}_h}{4},
 \qquad \nu(h)=\lambda_{\min}(Q_h).
\end{equation}
The operator $4Q_h$ is the conformal Laplacian. Its relation to the
spectrum of the untwisted spin Dirac operator is classical
\cite{Hijazi1986}; here we also control the twisting curvature term.

\begin{theorem}\label{thm:conformal}
Under the sphere-bundle and total-metric hypotheses of
Theorem~\ref{thm:spectral}, suppose $n\ge3$ and
$\nu(h_b)\ge\nu_0>0$ for every $b\in B$. Then
\begin{equation}\label{eq:conformal-bound}
 \Karea(E,g)\le\frac{\pi n(n-1)}{\nu_0}.
\end{equation}
\end{theorem}

\begin{proof}
We replace the analytic step \eqref{eq:gap}; the topological and
compact-reduction arguments are unchanged. Work first on a closed
fibre with a vertical $\Spin^c$ structure whose determinant connection
is flat along that fibre. Suppose the twisting curvature satisfies
$\kappa<2\nu_0/[n(n-1)]$.
Choose a positive lowest eigenfunction $v$ of $Q_h$, so
$Q_hv=\nu(h)v$. Such a positive eigenfunction exists by the ground-state
and strong maximum principles on a closed connected manifold.
For $h'=v^{4/(n-2)}h$, the conformal scalar-curvature formula gives
\[
 \frac{\operatorname{Sc}_{h'}}4
 =v^{-(n+2)/(n-2)}Q_hv
 =\nu(h)v^{-4/(n-2)}.
\]
The operator comass of the same curvature two-form scales as
\[
 \|F_\nabla(x)\|_{h'}=v(x)^{-4/(n-2)}\|F_\nabla(x)\|_h,
\]
where these are the pointwise comasses.
Consequently the zero-order part of the twisted Dirac square for $h'$
is bounded below by
\begin{equation}\label{eq:conformal-gap}
 v^{-4/(n-2)}\left(\nu(h)-\frac{n(n-1)}2\kappa\right)>0.
\end{equation}
The determinant term is still zero. Thus $D_{h',V}$ has no kernel.

Put $f=v^{2/(n-2)}$ and let $I$ be the usual identification of spinor
bundles for $h'=f^2h$, tensored with the identity on $V$.
The conformal covariance formula is
\begin{equation}\label{eq:conformal-dirac}
 D_{h',V}\bigl(f^{-(n-1)/2}I\psi\bigr)
 =f^{-(n+1)/2}I D_{h,V}\psi.
\end{equation}
For the untwisted formula see \cite[\S2]{Grosse2011}.
It persists for the fixed twisting and determinant connections:
in a local frame their additional Dirac terms are Clifford contraction
with a connection one-form, and the orthonormal-frame change
$e_i\mapsto f^{-1}e_i$ scales that contraction by $f^{-1}$.
There is no additional derivative of the twisting connection.
Equation~\eqref{eq:conformal-dirac} identifies the two kernels, so the
original operator $D_{h,V}$ has no kernel either. The trivial twist
satisfies the same argument with $\kappa=0$.

Only a fibrewise choice of $v$ was used. The original metric defines
a smooth Dirac family whose kernel is now zero on every fibre, so
Lemma~\ref{lem:two-indices} applies to that family. No smooth global
choice of the positive eigenfunctions is needed.
Use Lemma~\ref{lem:brauer} to obtain the specified determinant and
\eqref{eq:brauer-return} to return to $B$. This gives the closed-base
vanishing statement at the threshold $2\nu_0/[n(n-1)]$.
Lemma~\ref{lem:double} preserves $\nu$ because its new fibres are
isometric to original fibres. Its framed gluing and the argument on
$E\times\R$ then prove \eqref{eq:conformal-bound} for all the detecting
bundles in Definition~\ref{def:karea}.
\end{proof}

The Rayleigh quotient gives, on the same fibre,
\[
 \lambda_{\min}(\Delta_h+\operatorname{Sc}_h/4)\le\nu(h).
\]
Thus whenever Theorem~\ref{thm:spectral} applies in dimension at least
three, the conformal estimate applies with a denominator at least as
large. Both spectral quantities equal $n(n-1)/(4R^2)$ on a round
sphere of radius $R$, so the strengthening retains the bound $4\pi R^2$.
The conformal change is used only to establish kernel vanishing;
it gives neither conformal invariance of K-area nor the operator
inequality $D_{h,V}^2\ge\nu_0-n(n-1)\kappa/2$.

The conformal condition also permits some metrics with negative scalar
curvature at a point. On the unit sphere let $z$ be its last coordinate
and set $h_t=e^{2tz}h_{\mathrm{round}}$. At the south pole,
\[
 \operatorname{Sc}_{h_t}=n(n-1)e^{2t}(1-2t)<0\qquad(t>1/2).
\]
Indeed, for our nonnegative Laplacian the conformal transformation law
is
\[
 \operatorname{Sc}_{e^{2u}h}=e^{-2u}\bigl(
 \operatorname{Sc}_h+2(n-1)\Delta_hu
 -(n-1)(n-2)|\dd u|_h^2\bigr).
\]
At that pole $\Delta z=nz=-n$ and $\dd z=0$.
Nevertheless $Q_{h_t}$ is positive: if
$w=e^{(n-2)tz/2}$, conformal covariance gives
\[
 \int_{S^n}\phi Q_{h_t}\phi\,\dd\operatorname{vol}_{h_t}
 =\int_{S^n}(w\phi)Q_{h_{\mathrm{round}}}(w\phi)
                         \,\dd\operatorname{vol}_{h_{\mathrm{round}}}>0
\]
for every nonzero $\phi$.
Since $w$ and the volume ratio are bounded above and below on the
compact sphere, this gives $\nu(h_t)>0$. A product bundle with this
fixed fibre metric satisfies Theorem~\ref{thm:conformal} although its
fibre scalar curvature is not everywhere positive.

\section{Normalization and scope}\label{sec:scope}

The strict threshold in Lemma~\ref{lem:holomorphic} is necessary: a
degree-one line bundle on an oriented two-sphere of area $a$ admits
a unitary connection with
\[
 F_\nabla=-\frac{2\pi i}{a}\,\dd A.
\]
To obtain such a connection, start with any unitary connection; the
difference between its curvature and the displayed form is an exact
imaginary-valued two-form, since their integrals agree, and adjust the
connection by an imaginary-valued one-form. Its first Chern number is
$1$ and its curvature norm is $2\pi/a$. Together with
Theorem~\ref{thm:area}, applied with base a point, this gives
$\Karea(S^2,h)=a$.

In particular, a round two-sphere of radius $R$ has K-area
$4\pi R^2$. More generally, replacing $g$ by $R^2g$ multiplies the
curvature comass by $R^{-2}$ and K-area by $R^2$. For odd total
dimension the same statement follows from \eqref{eq:odd} by a
rescaling of the real-line coordinate. Thus the unit-radius
condition in Corollary~\ref{thm:round} fixes the scale of its constant.

The two-sphere proof gives actual descent of the detecting bundle
after orienting the vertical tangent bundle. The family-index proof
only needs descent of Chern-character classes in real cohomology on
the compact reduction. It does not assert integral K-theoretic descent
or vanishing of torsion. For noncompact spaces, the parallel frame at
infinity is used explicitly in both proofs. Neither proof establishes
the same conclusion for a definition allowing arbitrary nontrivial
flat holonomy at infinity.

\section*{Acknowledgements}
The author thanks Professor Bobo Hua for his support.

\section*{Use of AI tools}
The author used GPT-6 Astra in this work and guided its reasoning.
The author takes full responsibility for the content of the paper.

\bibliographystyle{abbrvnat}
\bibliography{references}
\end{document}